\documentclass[10pt,reqno]{amsart}
\usepackage{amssymb,amsfonts}
\usepackage[all,arc]{xy}
\usepackage{enumerate}
\usepackage{mathrsfs}
\usepackage{geometry}
\usepackage{amsmath}
\usepackage{latexsym, bm}
\usepackage{indentfirst}
\usepackage{CJK}
\usepackage{amsthm}
\usepackage{hyperref}
\usepackage{cite}
\usepackage{lipsum}

\DeclareMathOperator{\tr}{tr}

\newcommand{\p}{\frac{\partial}{\partial z^i}}

\newcommand{\om}{\omega}

\newcommand{\vp}{\varphi}
\newcommand{\pa}{\partial}
\newcommand{\bap}{\bar{\partial}}

\newcommand{\Om}{\Omega}

\newcommand{\al}{\alpha}

\newcommand{\tl}{\tilde}
\newcommand{\wtl}{\widetilde}
\newcommand{\ov}{\overline}

\newcommand{\La}{\Lambda}
\newcommand{\lt}{\left}

\newcommand{\ag}{\langle}
\newcommand{\rg}{\rangle}
\newcommand{\wg}{\wedge}
\newcommand{\bz}{\bar{z}}

\newcommand{\bj}{\bar{j}}
\newcommand{\bl}{\bar{l}}

\newcommand{\bq}{\bar{q}}
\newcommand{\so}{\sqrt{-1}}
\newcommand{\bk}{\bar{k}}

\newtheorem{thm}{Theorem}[section]
\newtheorem{cor}[thm]{Corollary}
\newtheorem{lem}[thm]{Lemma}
\newtheorem{prop}[thm]{Proposition}

\theoremstyle{definition}
\newtheorem{defn}[thm]{Definition}
\newtheorem{remk}[thm]{Remark}
\newtheorem{exam}[thm]{Example}

\numberwithin{equation}{section}

\begin{document}
\title{Chern-Ricci curvatures, holomorphic sectional curvature and Hermitian metrics}
\author{Haojie Chen, Lingling Chen, Xiaolan Nie}
\address{Department of Mathematics, Zhejiang Normal University, Jinhua Zhejiang, 321004, China}
\email{chj@zjnu.edu.cn,\ chenll0716@126.com, \ nie@zjnu.edu.cn }
\date{}
\maketitle

\begin{abstract}
We present some formulae related to the Chern-Ricci curvatures and scalar curvatures of special Hermitian metrics. We prove that a compact locally conformal K\"{a}hler manifold with constant nonpositive holomorphic sectional curvature is K\"{a}hler.  We also give examples of complete non-K\"{a}hler metrics with pointwise negative constant but not globally constant holomorphic sectional curvature, and complete non-K\"{a}hler metric with zero holomorphic sectional curvature and nonvanishing curvature tensor.\\

\end{abstract}

\section{Introduction}
\let\thefootnote\relax\footnotetext{Supported in part by NSFC (Grant No. 11801516) and ZJNSF (Grant No. LY19A010017)}
This note mainly concerns Hermitian manifolds with constant or pointwise constant holomorphic sectional curvature. For a general Hermitian manifold $(M, \om)$, the (Chern) holomorphic sectional curvature $H$ is defined by $$H(X)=R(X,\bar{X},X,\bar{X})/|X|^4,$$ where $R$ is the curvature tensor of the Chern connection and $X\in T_p^{1,0}(M)$(\cite{KN}\cite{Z}). The holomorphic sectional curvature plays a fundamental role in complex geometry. Complete K\"{a}hler manifolds with constant holomorphic sectional curvature are called complex space forms\cite{Z}. They are natural analogue of complete Riemannian manifolds with constant sectional curvature. It is known (\cite{H}, \cite{I}) that a simply connected complex space form is holomorphically isometric to the complex projective space $\mathbb{CP}^n$, the complex hyperbolic space $\mathbb{B}^n$ or $\mathbb{C}^n$.

 In \cite{BG}, Balas-Gauduchon prove that a compact Hermitian surface with constant nonpositive holomorphic sectional curvature must be K\"ahler. In higher dimension, it is known that there are examples of compact non-K\"{a}hler manifolds with $H=0$ (e.g. the Iwasawa manifold, see \cite{B}). A natural question is: if  $(M,\om)$ is an $n$-dimensional compact Hermitian manifold with constant (or pointwise constant) negative holomorphic sectional curvature and $n\geq 3$, is $(M,\om)$ still K\"{a}hler? 

The method of Balas-Gauduchon depends heavily on $n=2$. When $n\geq 3$, we restrict ourselves to locally conformal K\"{a}hler manifolds. There are rich examples of locally conformal K\"ahler manifolds, including the elliptic surfaces, diagonal Hopf manifolds and flat principal circle bundles over a compact Sasakian manifold \cite{DO}. We prove the following result.

\begin{thm} \label{thm1} Let $(M,\om)$ be a compact locally conformal K\"ahler manifold with constant nonpositive holomorphic sectional curvature. Then $(M, \om)$ is K\"ahler. In particular, the universal cover of $(M,\om)$ is the complex hyperbolic  space $\mathbb{B}^n$ or $\mathbb{C}^n$. \end{thm}

The proof of Theorem \ref{thm1} is based on a relation between the first and second Chern-Ricci curvatures for locally conformal K\"{a}hler metrics. In \cite{LY2}, Liu-Yang systematically study a variety of Ricci curvatures on a Hermitian manifold. Among other results, they derive explicit relations between all kinds of Ricci curvatures on general Hermitian manifolds. In the locally conformal K\"{a}hler case, we can get a simpler formula (see Proposition \ref{l3}). Then we are able to reduce the theorem to the conformally K\"{a}hler case. Actually we prove the following result under the more general pointwise constant condition.
\begin{thm} \label{lck0} Let $(M,\om)$ be a compact locally conformal K\"ahler manifold with pointwise nonpositive constant holomorphic sectional curvature. Then $(M, \om)$ is globally conformal K\"ahler. 
\end{thm}
We remark that Vaisman \cite{Vai} has proved that a locally conformal K\"{a}hler metric with pointwise constant (Chern) sectional curvature is either globally conformal K\"ahler or has vanishing first Chern class. The constancy of sectional curvature is of course stronger than the constancy of holomorphic sectional curvature. For example, the sectional curvatures of $\mathbb{CP}^n$ and $\mathbb{B}^n$ $(n\geq 2$) are not pointwise constant.

An important class of locally conformal K\"ahler manifolds is called Vaisman manifolds, whose Lee form is parallel with respect to the Levi-Civita connection. It is shown\cite{PPS} that a Vaisman metric on a compact manifold must be Gauduchon. Then we obtain
\begin{cor} \label{cor1.3}
A compact Vaisman manifold with pointwise nonpositive constant holomorphic sectional curvature is K\"ahler.
\end{cor}
Considering the constancy of holomorphic sectional curvature, a natural question is: does the pointwise constancy of $H$ imply the global constancy? 

When $\om$ is K\"ahler and $n\geq 2$, it is always true by the Schur's Lemma, as $\om$ is K\"{a}hler-Einstein and $ H$ is constant multiple of the scalar curvature (see Proposition \ref{bal}). If $\om$ is non-K\"{a}hler, we construct counterexamples showing that the Schur type result does not hold in general  (example \ref{ex1}). 
\begin{prop} \label{thm3} There exist non-K\"{a}hler, conformally flat metrics on $\mathbb{C}^n (n\geq 2)$ with pointwise negative constant (or pointwise positive constant) but not globally constant holomorphic sectional curvature. In the negative case, the metric is complete. 
\end{prop}

\begin{remk} If the holomorphic sectional curvature is defined using the Levi-Civita connection,  it is Gray-Vanheche \cite{GV} who first discovered that the Schur type result does not hold in non-K\"{a}hler setting.  Since the Levi-Civita connection coincides with the Chern connection if and only if the metric is K\"{a}hler, our result are obviously different from theirs. We refer to \cite{GV}\cite{SS}\cite{SK} and the references therein for more results and development in that direction.  Also see \cite{LU} for some recent results on almost K\"{a}hler 4-manifolds with constant nonnegative (Chern) holomorphic sectional curvature.
\end{remk}

Using the conformal change technique, we also show the following (see example \ref{ex2}).
 \begin{prop} \label{prop1.5}
There exist complete non-K\"{a}hler, conformal K\"{a}hler metric on $\mathbb{C}^n (n\geq 2)$ with zero holomorphic sectional curvature but nonvanishing curvature tensor.
\end{prop}

To the authors' knowledge, all the previously known examples of non-K\"{a}hler manifolds with $H=0$ have vanishing curvature (the quotient of complex Lie groups, see \cite{Bo}). Our example implies that the holomorphic sectional curvature does not necessarily determine the curvature tensor of a Hermitian metric. Also, it  shows that the compactness condition in Theorem \ref{thm1} can not be replaced by completeness in the $H=0$ case. It would be an interesting question to study whether there exist similar examples as in Propostion \ref{thm3} and Proposition \ref{prop1.5} on compact manifolds.

We also discuss another notion of special Hermitian metrics, which is called the $k$-Gauduchon metric. It is introduced by Fu-Wang-Wu in \cite{FW} as a generalization of the Gauduchon metric. A $k$-Gauduchon metric is a Hermitian metric satisfying
$$\so \pa\bap\om^k\wg w^{n-k-1}=0.$$ In particular, a pluriclosed metric (i.e. $\pa\bap\om=0$) is 1-Gauduchon, while an Astheno-K\"{a}hler metric (i.e. $\pa\bap\om^{n-2}=0$) is $(n-2)$-Gauduchon. 
We have the following characterization.
\begin{prop}
Let $(M,\om)$ be an $n$-dimensional Hermitian manifold $(n\geq 3)$ and $k$ be an integer such that $1\leq k\leq n-1$. Then the following are equivalent:
\begin{itemize}
\label{kg01}
\item[(1)] $\om$ is $k$-Gauduchon;\\
\item[(2)]$s-\hat{s}=\dfrac{k-1}{n-2}|\pa^*\om|^2+\dfrac{n-1-k}{n-2}|\pa\om|^2$.
\end{itemize}
\end{prop}

\begin{remk} Another characterization in terms of scalar curvatures of the Bismut connection (also called the Strominger connection \cite{ZZ}) is obtained in \cite{FZ}. A direct corollary is that  if $(M,\om)$ is $k$-Gauduchon, then $s\geq \hat{s}$. 
\end{remk}
 We mention that recently there have been breakthroughs on negative or positive holomorphic sectional curvatures on K\"{a}hler manifolds. A conjecture of Yau on negative holomorphic sectional curvature and ampleness of the canonical bundle has been confirmed \cite{TY} \cite{WY1}(see also \cite{DT}, \cite{HLW1}, \cite{HLW2}, \cite{HLWZ}, \cite{WWY}, \cite{WY2}, \cite {YZ19} e.t.c. for some earlier work and further generalizations).  Also, another conjecture of Yau which says  a compact K\"{a}hler manifold with positive holomorphic sectional curvature must be projective and rationally connected is proved by  Yang \cite{Y18} (see \cite{NZ} for related work). A compact Hermitian manifold with negative holomorphic sectional curvature is Kobayashi hyperbolic \cite{GR}. Its canonical bundle is also conjectured to be ample (see \cite{L}\cite{YZ19} for some recent progress). \\

The structure of the paper is as follows. In section 2, we give some background of Hermitian geometry.  In section 3, we prove Theorem \ref{thm1}, Theorem \ref{lck0} and give the examples mentioned above (see examples \ref{ex1}, \ref{ex2}, \ref{ex3}).  Finally, we study some properties of the $k$-Gauduchon metric and prove Proposition \ref{kg01} (see Proposition \ref{kg1}).\\

\noindent \textbf{Acknowlegements}. The first-named author and the third-named author would like to thank Professors Jixiang Fu and Xiaokui Yang for helpful discussions. They are also grateful to Professors Jiaping Wang and Fangyang Zheng for their encouragement and helpful suggestions. 

\section{Torsion one form and Ricci curvatures}

In this section, we give some background materials in Hermitian geometry. We will present some formulae related to the Chern connection.  The readers are referred to \cite{D},  \cite{LY1}, \cite{LY2}, \cite{V}, \cite{YZ18} and \cite{Z}  for more details. We remark that we do not make use of any good coordinates in our discussion. We will use Einstein summation notation throughout the paper.

\subsection{Operators on Hermitian manifolds}\hfill\\

Let $(M,g)$ be a $2n$-dimensional Riemannian manifold. Write $g=g_{ij}dx^i dx^j$, where $(x^1,x^2, ..., x^{2n})$ is a local coordinate on $M$. Denote $(g^{ij})$ the inverse matrix  of $(g_{ij}),1\leq i,j\leq 2n$. Then $g$ induces an inner product $\langle, \rangle$ on the cotangent bundle $T^*M$ by $\langle dx^i, dx^j \rangle=g^{ij}$. Let $\Lambda^kT^*M, \  1\leq k\leq 2n$ be the bundle of real $k$-forms. The inner product induced by $g$ on $\Lambda^kT^*M$ is:
 \begin{align}
 \label{in}\langle\alpha_1\wedge...\wedge \alpha_k, \beta_1\wedge...\wedge\beta_k\rangle=\det (\langle\alpha_i, \beta_j\rangle),\ \ \ \alpha_i, \beta_j\in T^*M.
 \end{align}
Equivalently,
 \begin{align} \langle \varphi , \psi \rangle=\dfrac{1}{k!}g^{i_1j_1}...g^{i_kj_k}\varphi_{i_1...i_k}\psi_{j_1...j_k},\label{expan}\end{align}
   for $$\varphi=\dfrac{1}{k!}\varphi_{i_1...i_k}dx^{i_1}\wedge...\wedge dx^{i_k},\ \ \ \ \
\psi=\dfrac{1}{k!}\psi_{j_1...j_k}dx^{j_1}\wedge...\wedge dx^{j_k},$$ where $\varphi_{i_1...i_k}$ is
 skew symmetric in $i_1,...,i_k$ and $\psi_{j_1...j_k}$ is skew symmetric in $j_1,...,j_k$. For $X\in TM$ and $\varphi \in \Lambda^kT^*M$, define the contraction (or interior product) $\iota_X \varphi \in \Lambda^{k-1}T^*M$ by
 $$\iota_X\vp \lt(X_1,..., X_{k-1}\right):= \vp(X, X_1, ...,X_{k-1}).$$ We have
 $$ \iota_X(\alpha_1\wedge...\wedge \alpha_k)=\sum_{i=1}^k(-1)^{i-1}\alpha_i(X)\alpha_1\wedge...\wedge \alpha_{i-1}\wedge\alpha_{i+1}\wedge...\wedge\alpha_k.$$
Denote $\widetilde{X}=g(X,\cdot)\in T^*M$ the metric dual of $X$, then we have
 \begin{align}
 \label{ip}\langle \iota_X\varphi , \psi\rangle =\langle \varphi, \widetilde{X}\wedge \psi\rangle
 \end{align}
 for $\varphi\in\wedge^{k+1}T^*M$ and $ \psi\in\wedge^{k}T^*M $.
 Indeed, if $\vp=\al_1\wedge...\wedge \alpha_{k+1},\ \psi=\beta_1\wedge...\wedge\beta_{k}, \al_i,\beta_j\in T^*M$,
 from (\ref{in}) we get
 \begin{align*}
 \ag\vp,\widetilde{X}\wg\psi\rg&=\sum_{i=1}^{k+1}(-1)^{i-1}\langle\al_i, \widetilde{X}\rangle\langle\al_1\wedge...\widehat{\al_i}\wedge...\wedge \alpha_{k+1},\beta_1\wedge...\wedge\beta_k\rangle\notag\\
&=\langle\sum_{i=1}^{k+1}(-1)^{i-1}\al_i(X)\al_1\wedge...\widehat{\al_i}\wedge...\al_{k+1}, \beta_1\wedge...\beta_k\rangle\notag \\
 &= \langle \iota_X\varphi , \psi\rangle.
 \end{align*}
 The general case follows by linear expansion.\\

Now assume that $(M,J)$ is a complex manifold. If the Riemannian metric $g$ satisfies $g(X,Y)=g(JX, JY)$ for all $X, Y \in TM$, then $g$ is called a Hermitian metric. Let $TM^{\mathbb{C}}=TM\otimes_{\mathbb{R}}{\mathbb{C}}$ be the complexified tangent bundle. Denote $h$ the $\mathbb{C}$-linear extension of $g$ to $TM^{\mathbb{C}}$. The fundamental $(1,1)$ form associated to $h$ is given by $\omega(X,Y)=g(JX, Y)$. Locally, $$\omega=\sqrt{-1}h_{i\bar{j}}dz^i\wedge d\bar{z}^j,$$where $(z^1, z^2, ..., z^n)$ is a local holomorphic coordinate and $h_{i\bar{j}}=h\left(\dfrac{\partial}{\partial z^i}, \dfrac{\partial}{\partial \bar{z}^j}\right)$. We also refer to $(M,\om)$ as a Hermitian manifold.  

Let ${\Omega}^{p,q}M$ be the space of $(p,q)$ form on $M$, $1\leq p,q\leq n.$
Extend the inner product $\langle \cdot,\cdot\rangle$ on $\La^kT^*M$  to $\Omega^{p,q}M$ in the following way:
$$\langle b\varphi_1+c\varphi_2 , \psi \rangle=b\langle \varphi_1 , \psi \rangle+c\langle \varphi_2 , \psi\rangle,$$
  $$\langle \varphi , b\psi_1+c\psi_2\rangle= \bar{b}\langle \varphi , \psi_1 \rangle+ \bar{c}\langle \varphi , \psi_2 \rangle,
  $$  for $b, c\in \mathbb{C}$ and  $\varphi_i, \psi_i\in \Lambda^kT^*M$. By this extension, for example,
  \begin{align*}
  \langle dz^i, dz^j\rangle=h^{i\bar{j}},
  \end{align*}
  where $(h^{i\bar{j}})$ is the (transposed) inverse matrix of $(h_{i\bar{j}})$(see \cite{LY2} for more details about the relation between $g_{i\bj}$ and $h_{i\bj})$).
   Also $\overline{\langle \varphi, \psi\rangle}=\langle \psi, \varphi\rangle$.
Write $(p,q)$ forms $\varphi$ and $\psi$ in local coordinates:
 $$ \varphi=\dfrac{1}{p!q!}\varphi_{i_1...i_p\overline{l_1}...\overline{l_q}}dz^{i_1}\wedge...\wedge dz^{i_p}\wedge d\bar{z}^{l_1}\wedge...\wedge d\bar{z}^{l_q}$$
  $$\psi=\dfrac{1}{p!q!}\psi_{j_1...j_p\overline{k_1}...\overline{k_q}}dz^{j_1}\wedge...\wedge dz^{j_p}\wedge d\bar{z}^{k_1}\wedge...\wedge d\bar{z}^{k_q},$$
 where $\varphi_{i_1...i_p\overline{l_1}...\overline{l_q}}$ is skew symmetric in $i_1, ..., i_p$ and skew symmetric in $l_1,..., l_q$ (similarly for $\psi_{j_1...j_p\overline{k_1}...\overline{k_q}})$. Then

\begin{align}
\langle \varphi, \psi\rangle=\dfrac{1}{p!q!}h^{i_1\bar{j_1}}...h^{i_p\overline{j_p}}h^{k_1\overline{l_1}}...h^{k_q\overline{l_q}}\varphi_{i_1...i_p\overline{l_1}...\overline{l_q}}\overline{\psi_{j_1...j_p\overline{k_1}..\overline{k_q}}
 }.\label{complex} \end{align}
This extends the formula (\ref{expan}). \\
Let $dv$ be the volume form of $g$, the total inner product is defined to be \begin{align*} (\varphi, \psi)=\int_M \langle \varphi, \psi\rangle dv.\end{align*}
Denote $|\vp|^2=\langle \vp, \varphi\rangle$.  The total norm is $$||\vp||^2=\int_M |\varphi|^2 dv.$$

The Hodge $*$ operator is the unique operator determined by $g$ satisfying the following:
  $$*:\Lambda^kT^*M\rightarrow \Lambda^{2n-k}T^*M, \ \ \ \ \varphi\wedge *{\psi}=\langle\varphi,\psi\rangle dv,$$ where $\varphi, \psi\in \Lambda^kT^*M$. It is extended $\mathbb{C}$-linearly to complex forms and satisfies:
 $$*:\Omega^{p,q}M\rightarrow\ \Omega^{n-q, n-p}M,\ \ \ \ \varphi\wedge *\overline{\psi}=\langle\varphi,\psi\rangle dv$$
  for $\varphi, \psi\in \Omega^{p,q}M$ and $dv=\dfrac{\omega^n}{n!}$.
Also, we have $$\overline{*\varphi}=*\overline{\varphi},\ \ \ \ **\varphi=(-1)^{p+q}\varphi,\ \ \ \  \langle *\varphi, *\psi\rangle=\langle\varphi,\psi\rangle.$$
The formal adjoint operators $\pa^*, \bar{\pa}^*$ are  given by \begin{align*} (\pa \varphi, \psi)=(\vp, \pa^* \psi),  \hspace*{.8cm} (\bap \varphi, \psi)=(\vp, \bap^* \psi). \end{align*} It is well known that $$\pa^*=-*\bap*, \ \ \ \bar{\pa}^*=-*\pa*$$ on compact Hermtian manifolds.
Write $\iota_j\varphi=\iota_{\frac{\partial}{\partial z^j}}\varphi$ and $\iota_{\bar{j}}\varphi=\iota_{\frac{\partial}{\partial\bar{z}^j}}\varphi$
 for convenience. It follows from (\ref{ip}) that
\begin{align}
\label{cw}\langle\varphi,\  dz^i\wedge\psi\rangle=\langle h^{j\bar{i}}\iota_j\varphi,\psi\rangle,\\
\notag \langle\varphi,\  d\bar{z}^i\wedge\psi\rangle=\langle h^{i\bar{j}}\iota_{\bar{j}}\varphi, \ \psi\rangle.
\end{align}
The Lefschetz operator $L:\Omega^{p,q}M\rightarrow\ \Omega^{p+1, q+1}M$ and its adjoint $\Lambda:\Omega^{p+1,q+1}M\rightarrow\ \Omega^{p, q}M$ are defined by
$$L\varphi=\omega\wedge\varphi, \ \ \ \ \langle L\varphi, \psi\rangle=\langle \varphi, \Lambda \psi\rangle.$$
From (\ref{cw}), in local coordinates, we get 
\begin{align} \Lambda=\sqrt{-1}h^{i\bar{j}}\iota_i\iota_{\bar{j}}.\label{contr}\end{align}
For $\vp\in\Om^{p,q}M $ with $p+q=k$, direct computation gives (see also \cite{V})$$
 \Lambda(\om\wg \vp)=(n-k)\vp+\om\wg (\Lambda \vp),$$ that is,
 \begin{align}
 \label{cm}[L,\Lambda]\vp=(k-n)\vp.
 \end{align}
Applying this equality repeatedly, we have
\begin{align*}
[L^r, \Lambda]\vp=[L^{r-s}, \La]L^s\vp+s(k-n+s-1)L^{r-1}\vp.
\end{align*}
In particular, for $s=r$,
\begin{align}
\label{cmr} [L^r, \Lambda]\vp=r(k-n+r-1)L^{r-1}\vp.
\end{align}

 \begin{defn} A $(p,q)$ form $\varphi$ is called primitive if $\Lambda\vp=0$.\end{defn}

\noindent For a primitive $\varphi\in \Om^{p,q}M$ with $p+q=k$, we have
 \begin{align}
 \label{eta}
  \Lambda(\om\wg \vp)&=(n-k)\vp,\ \ \ \ \ \ \ \ |\vp\wg \om|^2=(n-k)|\vp|^2,\\[.4em]
\label{lr} &\La L^r\varphi=r(n-k-r+1)L^{r-1}\vp. \\ \notag
\end{align}

\subsection {Torsion 1-form} \hfill\\

Let $\nabla$ be the Chern connection of a Hermtian manifold $(M, J, h)$. It is the unique connection such that $\nabla J=0, \nabla h=0$ and the torsion tensor
\begin{align*}T(X,Y)=\nabla_XY-\nabla_YX-[X,Y], \ \ \ \ \ \ X, Y\in TM.
\end{align*}
has vanishing $(1,1)$ part. Using local coordinates, the Christoffel symbols $\Gamma_{ij}^k$ and the torsion $T$ of $\nabla$ are given by (see e.g. \cite{TW})
$$\Gamma_{ij}^k=h^{k\bar{l}}\partial_i h_{j\bar{l}}, \ \ \ \ T_{ij}^k=\Gamma_{ij}^k-\Gamma_{ji}^k=h^{k\bl}T_{ij\bl}.$$ As $\omega=\sqrt{-1}h_{i\bar{j}}dz^i\wedge d\bar{z}^j$, then 
\begin{align}
\label{paom}\pa{\omega}=\dfrac{\sqrt{-1}}{2}T_{ij\bk}dz^i\wedge dz^j\wedge d\bz^k.
\end{align}
The torsion 1-form of the Chern connection is a $(1,0)$-form defined by
$\tau=\tau_i dz^i$ with $\tau _i=T_{ik}^k=h^{k\bl}T_{ik\bl}$. Using (\ref{contr}) and (\ref{paom}), direct computation gives 
 \begin{align}\label{tau0}
\tau=\Lambda \pa\om=\sum_{k=1}^n T_{ik}^k dz^i.
\end{align}
The following result is well known on Hermitian manifolds. 
\begin{lem}
	Let $(M,\om)$ be an $n$-dimensioanl Hermitian manifold.  Then
\begin{align}
\pa\om^{n-1}=\tau\wedge\om^{n-1}.\label{lee}
\end{align}
\end{lem}
\begin{proof}
	Define $\al_0=\pa\om-\frac{1}{n-1}L(\Lambda\pa\om)$. Then $\La\al_0=0$ by (\ref{eta}). This is equivalent to $L^{n-2}\al_0=0$ (see \cite{V}), i.e.
	\begin{align*}
(\pa\om-\frac{1}{n-1}L\Lambda\pa\om)\wedge\om^{n-2}=0.
	\end{align*}
	Then
	\begin{align*}
	\pa\om^{n-1}=&(n-1)\pa\om\wg\om^{n-2}\\
	=&L(\La\pa\om)\wg\om^{n-2}\\
	=&\tau\wg\om^{n-1},\end{align*} where the last equality follows from (\ref{tau0}).
	\end{proof}

The following lemma will be used frequently.  It may be obtained by Bochner formula on Hermitian manifolds (see e.g. \cite{G1}, \cite{LY2}). We provide a naive proof here.
\begin{lem}
Let $(M,\om)$ be a compact Hermitian manifold and $\tau=T_{ik}^k dz^i$ be the torsion 1-form of the Chern connection. Then
 \begin{align}
\label{tau1}
\tau=\Lambda \pa\om=-\sqrt{-1}  \bar{\pa}^*\om
\end{align}
\end{lem}
\begin{proof}
Let $\vp$ be any $(1,0)$ form. As $*\om=\dfrac{\om^{n-1}}{(n-1)!}$, by (\ref{lee}),
\begin{align*}
\ag\vp, \bap^*\om\rg\dfrac{\om^n}{n!}=& \vp\wg\ov{\tau}\wg\frac{\om^{n-1}}{(n-1)!}\\=&\ag\vp\wg \ov{\tau},\om\rg\dfrac{\om^n}{n!}\\
=&\ag\vp,\so \tau\rg\dfrac{\om^n}{n!}
\end{align*}
Then we see $\tau=-\so\bap^*\om$. The first equality in (\ref{tau1}) follows from (\ref{tau0}).
\end{proof}\noindent Recall that a Hermitian metric is called balanced if $d\om^{n-1}=0$, namely, $\tau=0$. We can easily get the following fact (also see \cite{B}).
\begin{cor} Let $(M,\om)$ be a compact Hermitian manifold. If the torsion 1-form $\tau$ is holomorphic, then $(M,\om)$ is balanced.
\end{cor}
\begin{proof}
By (\ref{tau1}),  $\bap{\tau}=-\sqrt{-1}\bar{\pa}\bar{\pa}^*\om=0$, then we have $||\bar{\pa}^*\om||^2=(\ov{\pa}\bar{\pa}^*\om,\om)=0$. So $\tau=-\sqrt{-1}\bar{\pa}^*\om=0$ and $\om$ is balanced.

\end{proof}

\subsection{Curvatures on Hermitian manifolds}\hfill\\

Let  $R_{i\bar{j}k\bar{l}}$ be the components of the curvature tensor of the Chern connection $\nabla$, then $$R_{i\bar{j}k\bar{l}}=-h_{p\bl}\pa_{\bj}\Gamma_{ik}^p.$$ Also, the following commutative relations hold (see e.g. \cite{N}, \cite{SW}):.
\begin{align}
\label{c1} R_{i\bar{j}k\bar{l}}-R_{k\bar{j}i\bar{l}}&=-\nabla_{\bar{j}}T_{ik\bar{l}}\\
	\label{c2}R_{i\bar{j}k\bar{l}}-R_{i\bar{l}k\bar{j}}&=-\nabla_iT_{\bar{j}\bar{l}k}\\
 	R_{i\bar{j}k\bar{l}}-R_{k\bar{l}i\bar{j}}&=-\nabla_{\bar{j}}T_{ik\bar{l}}-\nabla_kT_{\bar{j}\bar{l}i}. \label{c3}
 \end{align}
The $i$-th Chern-Ricci form $\rho^{(i)}=\sqrt{-1}\rho^{(i)}_{i\bj}dz^i\wedge d\bar{z}^j,\  1\leq i\leq 4 $ are defined by
 \begin{align*}
 	\rho_{i\bj}^{(1)}&=h^{k\bl}R_{i\bj k\bl},\ \ \ \ \rho_{i\bj}^{(2)}=h^{k\bl}R_{k\bl i\bj},\\
 	\rho_{i\bj}^{(3)}&=h^{k\bl}R_{i\bl k\bj},\ \ \ \ \rho_{i\bj}^{(4)}=h^{k\bl}R_{k\bj i\bl}.
 \end{align*}
The two scalar curvatures are
 \begin{align*}
 	s=h^{i\bj}h^{k\bl}R_{i\bj k\bl},\ \ \ \ \hat{s}=h^{i\bl}h^{k\bj}R_{i\bj k\bl},
 \end{align*}
where $s$ is the usual Chern scalar curvature and $\hat{s}$ is a Riemannian type scalar curvature. Both $s$ and $\hat{s}$ are real. 
\begin{prop} [\cite{G1}]
Let $(M,\om)$ be a compact Hermitian manifold, then
\begin{align}\label{scal2}
s-\hat{s}=\langle \pa\pa^* \om, \om\rangle=d^*\tau+|\tau|^2.
\end{align}
\end{prop}
\begin{proof}
By (\ref{tau1}),
\begin{align}
\label{pps}\pa\pa^*\om=-\so \pa_iT_{\bar{j}\bar{l}}^{\ \bl}dz^i\wg d\bz^j.
\end{align}
Then from (\ref{c2}) we have
 \begin{align}
\label{scal} s-\hat{s}=-h^{i\bj}\pa_iT_{\bar{j}\bar{l}}^{\ \bl}=\langle \pa\pa^* \om, \om\rangle.
 \end{align} As $\tau=-\so\bap^*\om$, the last equality of (\ref{scal2}) follows from the lemma below.
\end{proof}
\begin{lem}
Let $(M,\om)$ be a compact Hermitian manifold and $\phi$ be a (1,0) form on $M$. Then
 \begin{align}
\label{star1f}\pa^*\phi=\ag\so \bap\phi, \om\rg+\ag\phi,\so\bap^*\om\rg.
 \end{align}
\end{lem}
\begin{proof}
 For any $(1,0)$ form  $\vp$, we have
\begin{align*}
\vp\wg *\ov{\phi}&=\ag\so \vp\wg\ov{\phi}, \om\rg\frac{\om^n}{n!}=\so \vp\wg\ov{\phi}\wg\frac{\om^{n-1}}{(n-1)!}
\end{align*}
Then we see $*\phi=-\so \phi\wg\dfrac{\om^{n-1}}{(n-1)!}$. By  (\ref{lee}) and (\ref{tau1}), we have
\begin{align*}
\pa^*\phi=&*\bap(\so\phi\wg\dfrac{\om^{n-1}}{(n-1)!})\\
=&*(\so(\bap\phi-\phi\wg\ov{\tau})\wg\dfrac{\om^{n-1}}{(n-1)!})\\
=&\ag\so\bap\phi,\om\rg+\ag\phi, \so\bap^*\om\rg.
\end{align*}
\end{proof}
To compare the Ricci curvatures, we introduce some notion first.
\noindent For each $1\leq i\leq n$, define local $(2,0)$ forms $\xi^i=\frac{1}{2}T^i_{jk}dz^j\wedge dz^k$ as in \cite{YZ18} (see more details there). The column vector of the torsion 2-forms $(\xi^i)$ is denoted by $\xi$. Denote \begin{align*}
^t\xi\wg h\bar{\xi}&=\frac{1}{4} h_{i\bar{l}}T^i_{jk}T^{\bar{l}}_{\bar{r}\bar{s}}dz^j\wedge dz^k\wedge d\bar{z}^r\wedge d\bar{z}^s
\end{align*}
Then $^t\xi\wg h\bar{\xi}$ is independent of the local coordinates. It is a global nonnegative $(2,2)$ form defined on $M$ and vanishes if and only if $T=0$ (see \cite{YZ18}). By (\ref{contr}),  we have
\begin{align}
\Lambda(^t\xi\wg h\bar{\xi})=\so h^{i\bar{j}} h^{p\bar{q}}T_{ik\bar{q}}T_{\bar{j}\bar{l}p}dz^k\wedge d\bar{z}^l \label{contr3}
\end{align}

The following relation is established in Theorem 4.1 in \cite{LY2}. We give a direct proof here.
\begin{prop}[\cite{LY2}]
	Let $(M,h)$ be a compact n-dimensional Hermitian manifold and the $(2,2)$ form $^t\xi\wg h\bar{\xi}$ is defined as above. Then
 \begin{align}
 \label{ricci2} \rho^{(1)}-\rho^{(2)}=\La(\so\pa\bap\om)+\pa\pa^*\om+\bap\bap^*\om-\La(^t\xi\wg h\bar{\xi})
 \end{align}
\end{prop}
\begin{proof}
 From (\ref{c3}), we have
 \begin{align}
 \label{riccif} \rho_{i\bj}^{(1)}-\rho_{i\bj}^{(2)}=h^{k\bl}(\nabla_kT_{\bl\bj i}-\nabla_{\bj}T_{ik\bl}).
 \end{align}
Recall that $\bap^*\om=\sum_{k=1}^n \sqrt{-1} T_{ik}^k dz^i$, so
 \begin{align}
 \label{dbar1} \bar{\pa}\bar{\pa}^*\om&=-\sqrt{-1}\sum_{k=1}^n \nabla_{\bj}T_{ik}^{\ k}dz^i\wg d\bz^j
 \end{align}
Then (\ref{ricci2}) follows from (\ref{riccif}) and the next lemma.
\end{proof}

\begin{lem}
	Let $(M,h)$ be a compact Hermitian manifold and $\om=\sqrt{-1} h_{i \bar j}dz^i\wg d{\bar z}^ j.$ Then
	\begin{align}
\so h^{i\bj}\nabla_{i}T_{\bj\bl k}dz^k\wg d\bz^l=
\Lambda(\so\pa\bar{\pa}\om)+\pa\pa^*\om-\La(^t\xi\wg h\bar{\xi})\label{L4}
	\end{align}
\end{lem}
\begin{proof}
Direct computation gives that
\begin{align*}
\Lambda(\so\pa\bar{\pa}\om)&=\so h^{i\bar{j}}\iota_i\iota_{\bar{j}}(\frac{1}{4}(\pa_pT_{\bq\bl k}-\pa_kT_{\bq\bl p})dz^p\wg dz^k\wg d\bz^q\wg d\bz^l)\\
&=\so h^{i\bj}(\pa_iT_{\bj\bl k}-\pa_kT_{\bj\bl i})dz^k\wg d\bz^l\\
&=\so h^{i\bj}(\nabla_iT_{\bj \bl k}+\nabla_kT_{\bl\bj i}+h^{r\bar{s}}T_{ik\bar{s}}T_{\bj\bl r})dz^k\wg d\bz^l.
\end{align*}
Then (\ref{L4}) follows from (\ref{pps}) and (\ref{contr3}).
\end{proof}

\section{LCK metrics and holomorphic sectional curvature}
In this section, we will  discuss locally conformal K\"ahler (LCK for short)  manifolds and constant(or pointwise constant) holomorphic sectional curvature. Then we prove Theorem \ref{thm1} and give examples of non-K\"{a}hler manifolds with constant or pointwise constant holomorphic sectional curvature.  

Let $(M,\om)$ be an $n$-dimensional Hermitian manifold and $n\geq 2$.  A Hermitian metric $(M,\om)$ is called locally conformal K\"ahler if  $$d\om=\theta\wedge \om$$ and $\theta$ is closed. This is equivalent to that $\om$ is locally conformal to a K\"ahler metric. The real 1-form $\theta$ is called the Lee form of the Hermitian metric.  If $\om$ is locally conformal K\"ahler,  by (\ref{eta}) and (\ref{tau1}), we have
\begin{align}
\label{lck}
\pa\om=\dfrac{1}{n-1}\tau\wedge \om.
\end{align}
Note that (\ref{lck}) is also valid on any compact Hermitian surfaces.
 We have the following simple observation.
\begin{prop}
	Let $(M,\om)$ be a compact Hermitian surface. If the second Chern-Ricci curvature $\rho^{(2)}$ is closed, then $(M,\om)$ is a K\"ahler surface. In particular, if $\rho^{(1)}=\rho^{(2)}$, then $\om$ is K\"ahler.
\end{prop}
\begin{proof}
	By Lemme 1.12 in \cite{G1}, we have
	\begin{align*}
	\rho^{(1)}-\rho^{(2)}=\pa\pa^*\om-\pa^*\pa\om.
	\end{align*}
	As $\rho^{(1)}$ is closed, if in addition $\rho^{(2)}$ is closed, then 
	\begin{align*}
	{\|\pa^*\pa\om\|}^2=(\pa\pa^*\pa\om, \pa\om)=0.
	\end{align*}
	It follows that $(\pa^*\pa\om, \om)=||\pa\om||^2=0,$ i.e. $\om$ is K\"ahler.
\end{proof}


For LCK  manifolds or  surfaces, we have the following relations between $\rho^{(1)}$ and $\rho^{(2)}$. \begin{prop}
\label{l3} 
Let $(M,\om)$ be a compact locally conformal K\"ahler manifold. Then
\begin{align}
\rho^{(1)}-\rho^{(2)}=\dfrac{1}{n-1}((\hat{s}-s)\om+n\pa\pa^*\om).
\end{align}
If $(M, \om)$ is a compact Hermitian surface, then \begin{align}
\label{l4} \rho^{(1)}-\rho^{(2)}=(\hat{s}-s)\om+\pa\pa^*\om+\bap\bap^*\om.
\end{align}
\end{prop}
\begin{proof}
If $\om$ is locally conformal K\"{a}hler or $n=2$, then  $$\pa \om=\frac{1}{n-1}\tau\wg \om.$$ In local coordinates,
\begin{align*}
\pa\om=\frac{\so}{2}T_{ij\bk}dz^i\wg dz^j\wg d\bz ^k,\\
\tau\wg\om=\dfrac{\so}{2}(h_{j\bk}T_{il}^{\ l}-h_{i\bk}T_{jl}^{\ l})dz^i\wg dz^j\wg d\bz ^k.
\end{align*}
Then we have
\begin{align}
\label{lcf1}(n-1)T_{ij\bk}=h_{j\bk}T_{il}^{\ l}-h_{i\bk}T_{jl}^{\ l}.
\end{align}
By (\ref{riccif}),
 \begin{align} \label{riccif2}\rho_{i\bj}^{(1)}-\rho_{i\bj}^{(2)}=h^{k\bl}(\nabla_kT_{\bl\bj i}-\nabla_{\bj}T_{ik\bl}).
 \end{align}
By (\ref{lcf1}), we get
\begin{align}
&\notag\so h^{k\bl}\nabla_kT_{\bl\bj i}dz^i\wg d\bz^j\\
\notag=&\frac{\so}{n-1}h^{k\bl}\nabla_k(h_{i\bj}T_{\bl\bar{s}}^{\ \bar{s}}-h_{i\bl}T_{\bj\bar{s}}^{\ \bar{s}})dz^i\wg d\bz^j
\\
\label{lcf2}=&\frac{1}{n-1}(\pa\pa^*\om-(s-\hat{s})\om).
\end{align}
For a locally conformal K\"ahler metric, since $d(\tau+\bar{\tau})=0$, we have
\begin{align*}
\pa\pa^*\om-\bap\bap^*\om=0.
\end{align*}
Combining (\ref{riccif2}), (\ref{lcf2}) and  $\bap\bap^*\om=-\so\nabla_{\bj}T_{ik}^{\ k}dz^i\wg d\bz^j$,  we get (\ref{l3}) and (\ref{l4}).
\end{proof}
\begin{remk} The relations between all kinds of Ricci curvatures on general Hermitian manifolds are obtained in Liu-Yang\cite{LY2}. We get the simpler formula (\ref{l3}) due to the locally conformal K\"ahler condition in our case .\end{remk}

Now we consider Hermitian metrics with constant holomorphic sectional curvature. For a Hermitian manifold $(M,\om)$, the holomorphic sectional curvature is defined by $$H(X)=R(X,\bar{X},X,\bar{X})/|X|^4,$$ where $X\in T_p^{1,0}(M)$ and $R$ is the curvature tensor of the Chern connection. It is well known that when $\om$ is K\"ahler, $H$ dominates the whole curvature tensor $R$. In particular, $H=c$ if and only if \begin{align}\label{symm} R_{i\bar{j}k\bar{l}}=\frac{1}{2}c(h_{i\bar{j}}h_{k\bar{l}}+h_{i\bar{l}}h_{k\bar{j}}).\end{align} In this case, $(M,\om)$ is isometric to a quotient of simply-connected complex space forms (see \cite{Z}). However, if $\om$ is not K\"ahler, (\ref{symm}) does not hold any more.  Let
$K$ be the symmetric part of the curvature tensor $R$ with components $$K_{i\bj k\bl}=\frac{1}{4}(R_{{i\bj k\bl}}+R_{{k\bj i\bl}}+R_{{i\bl k\bj}}+R_{{k\bl i\bj}}).$$ In \cite{B}, Balas proved the following proposition. For reader's convenience, we give a proof here.
\begin{prop}[\cite{B}] \label{bal}
For a Hermitian manifold $(M,\om)$, the holomorphic sectional curvature $H=c$ at a point if and only if \begin{align} \label{curvature} K_{i\bar{j}k\bar{l}}=\frac{1}{2}c(h_{i\bar{j}}h_{k\bar{l}}+h_{i\bar{l}}h_{k\bar{j}}).\end{align} 
\end{prop}
\begin{proof}
For $p\in M$, let $X=X^i \p\in T^{1,0}_pM$. Then $H=c$ at $p$ if and only if
$$R_{i\bar{j}k\bar{l}}X^i\bar{X}^jX^k\bar{X}^l=c|X|^4.$$
We  rewrite the above equality to get \begin{align*}
K_{i\bar{j}k\bar{l}}X^i\bar{X}^jX^k\bar{X}^l=c(h_{i\bj}X^i\bar{X}^j)^2=\frac{1}{2}c(h_{i\bj}h_{k\bl}+h_{i\bl}h_{k\bj})X^i\bar{X}^jX^k\bar{X}^l.
\end{align*}
Note that both $K_{i\bj k\bl} $ and $h_{i\bj}h_{k\bl}+h_{i\bl}h_{k\bj}$ are symmetric in $i$, $k$ and also symmetric in $j$ and $l$. Thus the above equality holds for any $X=X^i \p$ if and only if (\ref{curvature}) is satisfied.
\end{proof}
\noindent Now we prove the following result (Theorem \ref{lck0}).
\begin{thm} \label{lckr} Let $(M,\om)$ be a compact Hermitian manifold. If $\om$ is locally conformal K\"ahler with  pointwise nonpositive constant holomorphic sectional curvature, then $(M, \om)$ is globally conformal K\"ahler.
\end{thm}
\begin{proof}
Let $H=c$, where $c$ is a nonpositive smooth function on $M$. Contracting (\ref{curvature}) with $\om$ to get (also see \cite{B})
\begin{align}
\label{sum1}\rho^{(1)}+\rho^{(2)}+2Re \rho^{(3)}&=2(n+1)c\om\\
\label{sum2} s+\hat{s}&=n(n+1)c.
\end{align}
As $\om$ is locally conformal K\"ahler, then $\pa\pa^*\om=\bap\bap^*\om$. From (\ref{c2}) and (\ref{pps}), we have
\begin{align} \label{rrr3}\rho^{(1)}=\rho^{(3)}+ \pa\pa^*\om.\end{align}
By Proposition \ref{l3},  (\ref{sum1}) and (\ref{sum2}), we get 
\begin{align}
\label{rr5}\rho^{(3)}=\frac{1}{4(n-1)}[((n+1)(n-2)c+2\hat{s})\om-(n-2)\pa\pa^*\om].
\end{align}
Also, from (\ref{rrr3}), $\rho^{(3)}$ is a real closed $(1,1)$ form. Take differential to (\ref{rr5}) to get
\begin{align}
\label{conformal} d((n+1)(n-2)c+2\hat{s})\om)=0.
\end{align}
Let $\varphi=(n+1)(n-2)c+2\hat{s}$. We will show that $\vp$ is either everywhere nonzero or $\vp\equiv 0$ (see also \cite{Vai}).

(1) If $\vp$ is everywhere nonzero, then $\om$ is conformally K\"ahler and we are done.

(2) Assume that $\vp(p)=0$ for some $p\in M$.  As $\om$ is locally conformal K\"ahler, we have
\begin{align}
\label{dvp}d\vp+\vp\theta=0
\end{align}
 and $d\theta=0$.  Then $\theta=du$ in a neighborhood of $p$  for some real function u.  By (\ref{dvp}), we have $e^ud\vp+\vp de^u=0$.   Then $e^u\vp$ is constant near $p$ and $\vp\equiv 0$  in a neighborhood of $p$. This implies that the set of points where $\vp$ equals zero is open and closed.  Thus $\vp\equiv 0$ on $M$.

We only need to discuss the case $\vp\equiv 0$ to finish the proof.

(a)If $c\equiv 0$, then $s=\hat{s}=0$ by (\ref{sum2}).  This implies that $\om$ is balanced, which means the Lee form $\theta=0$. As $\om$ is also locally conformal K\"{a}hler,  we get $d\om=0$ and $\om$ is K\"{a}hler.

(b) If $c\leq 0$ and $c<0$ at least at one point on $M$.  As $\hat{s}=-\frac{1}{2}(n+1)(n-2)c$, by (\ref{sum2}),  $$s=\frac{1}{2}(n+1)(3n-2)c\leq 0\leq \hat{s}.$$ It follows that $$\int_M\hat{s}dv> 0>\int_Ms dv,$$ which is a contradiction by Proposition \ref{scal2}. 

Thus either $\om $ is K\"{a}hler or $\vp$ is everywhere nonzero. Both imply that $\om$ is globally conformal K\"{a}hler.

\end{proof}

By the uniqueness of Gauduchon metric in each conformal class of a Hermitian metric, we also get  the following corollary.

\begin{cor} \label{lcg} Let $(M,\om)$ be a compact Hermitian manifold. If $\om$ is locally conformal K\"ahler and Gauduchon with pointwise nonpositive constant holomorphic sectional curvature, then $(M, \om)$ is K\"ahler. 
\end{cor}

In particular, any compact Vaisman manifold with pointwise nonpositive constant holomorphic sectional curvature is K\"{a}hler (Corollary \ref{cor1.3}) as it is Gauduchon (\cite{PPS}).

Next we consider holomorphic sectional curvature of globally conformal K\"{a}hler metric. Assume $\tilde{\om}=e^f\om$, where $f$ is a real smooth function. Denote $\widetilde{R}, \ \widetilde{\rho}^{(i)}, \ \widetilde{s}$ and $\ \widetilde{\hat{s}}$ the curvature tensor, Chern-Ricci curvatures and scalar curvatures of $\tilde{\om}$. Direct computation gives that
\begin{align}
\label{313}
\wtl{R}_{i\bj k\bl}=e^{f}(R_{i\bj k\bl}-\pa_i\pa_{\bj}f h_{k\bl}).
\end{align}
Then we obtain
\begin{align}
&\wtl{\rho}^{(1)}=\rho^{(1)}-n\so\pa\bap f,   \notag\\ \notag
&\wtl{\rho}^{(2)}=\rho^{(2)}-\tr_{\om}(\so\pa\bap f)\om,\\
&\wtl{\rho}^{(3)}=\rho^{(3)}-\so\pa\bap f.
\end{align}
It follows that
\begin{align}
\label {ssh2}\wtl{s}-\wtl{\hat{s}}=e^{-f}(s-\hat{s})-(n-1)\tr_{\tl{\om}}(\so\pa\bap f).
\end{align}
We call $\wtl{\om}=e^f\om$ a (globally) conformal K\"{a}hler metric if $\om$ is K\"{a}hler. We prove the following result.

\begin{prop}\label{gck}
 Let $(M,\wtl{\om})$ be a compact conformal K\"ahler manifold. If the holomorphic sectional curvature is nonpositive constant, then $(M, \wtl{\om})$ is K\"ahler.
\end{prop}
\begin{proof}
Denote $\widetilde{H}=c$ the holomorphic sectional curvature of $\wtl{\om}$, where $c$ is a nonpositive constant. The same argument as in Theorem \ref{lckr} gives \begin{align*} 
d((n+1)(n-2)c+2\tilde{\hat{s}})\tilde{\om})=0,\\
\wtl{s}+\wtl{\hat{s}}=n(n+1)c.
\end{align*}
As $\wtl{\om}=e^f\om$  with $d\om=0$, we have
\begin{align}
((n+1)(n-2)c+2\wtl{\hat{s}})e^{f}=A
\end{align} for some constant $A$. Then
\begin{align*}
\wtl{s}-\wtl{\hat{s}}&=n(n+1)c-2\wtl{\hat{s}}\\
&=2(n-1)(n+1)c-Ae^{-f}.
\end{align*}
Putting into (\ref{ssh2}), we get
\begin{align}\label{A}
-(n-1)\Delta f=2(n-1)(n+1)c e^{f}-A,
\end{align}
where $\Delta f=\tr_{\om }\so\pa\bap f.$  Let $dv$ be the volume form of  $\om$, then  $$\int_M\Delta fdv=\int_M\so\pa\bap f\wg\frac{\om^{n-1}}{(n-1)!}=0.$$
If $c=0$, then by (\ref{A}), we see $A=0$ and $f$ is a constant. \\
If $c<0$, first we have
\begin{align}
\label{del1}\wtl{\Delta}f=e^{-f}\Delta f=-\Delta e^{-f}+|\pa f|^2_{\tl{\om}}
\end{align}
where $\wtl{\Delta} f=\tr_{\wtl\om }\so\pa\bap f.$ From (\ref{A}), we get
\begin{align}
\label{del2}2(n-1)(n+1)c-Ae^{-f}=(n-1)(\Delta e^{-f}-|\pa f|^2_{\tl{\om}})
\end{align}
Integrate both sides of (\ref{A}) with respect to the K\"{a}hler metric $\om$, we get
\begin{align*}
&A\ vol(M)=2(n-1)(n+1)c\int_Me^f dv.\end{align*} 
As $c< 0$, we see $A<0$. Then integrate both sides of (\ref{del2}) with respect to $\om$ to get 
\begin{align*} A\int_M e^{-f}dv\geq 2(n-1)(n+1)c \ vol(M)\end{align*}
By Cauchy-Schwarz inequality and the above equations, we have
\begin{align*}
(vol(M))^2\leq \int_M e^fdv \int_M e^{-f}dv\leq (vol(M))^2.
\end{align*}
The equality holds if and only if $f$ is a constant. Therefore in either case, we derive that $\om$ is K\"{a}hler.
\end{proof}

\begin{proof}[Proof of Theorem \ref{thm1}] It follows from Theorem \ref{lckr}, Proposition \ref{gck} and the classical result for complex space forms\cite{H}\cite{I}. \end{proof}

We use Proposition \ref{bal} and the conformal trick to construct complete non-K\"{a}hler metrics on $\mathbb{C}^n$ with pointwise constant but not globally constant holomorphic sectional curvature.  We also give an example of  complete non-K\"{a}hler metrics on $\mathbb{C}^n$ with zero holomorphic sectional curvature and nonvanishing curvature tensor. This gives the proof of Proposition \ref{thm3} and Proposition \ref{prop1.5} .

\begin{exam}
\label{ex1}
Let $\om=\sum_{i=1}^n\sqrt{-1}dz^i\wedge d\bar{z}^i$ be the flat Euclidean metric on $\mathbb{C}^n$ with curvature $R=0$. Let $f=c|z|^2=c\sum_{i=1}^n |z_i|^2$, where $c$ is a nonzero real number and $\tilde{\om}=e^f\om$. Then $\partial_i\pa_{\bj} f=c\delta_{i j}$ and by (\ref{313}), $\tilde{R}_{i\bj k\bl}=-ce^f\delta_{ij}\delta_{kl}$. The symmetric curvature tensor $\wtl{K}$ is

\begin{align*}
\wtl{K}_{i\bj k\bl}&=\dfrac{1}{4}(\tilde{R}_{i\bj k\bl}+\tilde{R}_{k\bj i\bl}+\tilde{R}_{i\bl k\bj}+\tilde{R}_{k\bl i\bj})\\
&=-\dfrac{ce^f}{2}(\delta_{ij}\delta_{kl}+\delta_{il}\delta_{kj})\\
&=-\dfrac{ce^{-f}}{2}(\tilde{h}_{i\bj}\tilde{h}_{k\bl}+\tilde{h}_{i\bl}\tilde{h}_{k\bj}).
\end{align*}
By Proposition \ref{bal}, the holomorphic sectional curvature $\wtl{H}=-ce^{-f}$ is pointwise constant but not globally constant. When $c>0$, we see that $\tilde{\om}=e^{c|z|^2}\om$ is complete. 
\end{exam}
\begin{exam}
\label{ex2}
Let $\om=\sqrt{-1}\pa\bap \log (1+|z|^2)$ be the restriction of the Fubini-Study metric on $\mathbb{C}^n$. Then $$h_{i\bj}=\dfrac{(1+|z|^2)\delta_{ij}-\bar{z}_iz_j}{(1+|z|^2)^2}.$$
Also, the holomorphic sectional curvature of $\om$ is constant 2 (\cite{Z}), so we have $$R_{i\bj k\bl}=h_{i\bj}h_{k\bl}+h_{i\bl}h_{k\bj}.$$
Let $f=2\log (1+|z|^2)$ and $\tilde{\om}=e^f\om$. As $\partial_i\pa_{\bj} f=2h_{i \bj}$, we have
$$\widetilde{R}_{i\bj k\bl}=e^f(h_{i\bj}h_{k\bl}+h_{i\bl}h_{k\bj}-2h_{i\bj}h_{k\bl})=e^f(h_{i\bl}h_{k\bj}-h_{i\bj}h_{k\bl}).$$
The symmetric curvature tensor is \begin{align*}
\widetilde{K}_{i\bj k\bl}=\dfrac{1}{4}(\tilde{R}_{i\bj k\bl}+\tilde{R}_{k\bj i\bl}+\tilde{R}_{i\bl k\bj}+\tilde{R}_{k\bl i\bj})=0.
\end{align*}
So the holomorphic sectional curvature of $\tilde{\om}$ is zero, but the curvature is nowhere vanishing. Also, $\tilde{\om}=\so ((1+|z|^2)\delta_{ij}-\bar{z}_iz_j)dz^i\wedge d\bz^j$ is complete on $\mathbb{C}^n$. 
\end{exam}
\begin{exam}
\label{ex3}
 let $\om=-\sqrt{-1}\pa\bap \log (1-|z|^2)$ be the Bergman metric on the unit ball $\mathbb{B}^n$. Then $$h_{i\bj}=\dfrac{(1-|z|^2)\delta_{ij}+\bar{z}_iz_j}{(1-|z|^2)^2}.$$
Also, the holomorphic sectional curvature of $\om$ is constant -2. Then we have $$R_{i\bj k\bl}=-(h_{i\bj}h_{k\bl}+h_{i\bl}h_{k\bj}).$$
Let $f=2\log (1-|z|^2)$ and $\tilde{\om}=e^f\om$. Similarly, 
$$\widetilde{R}_{i\bj k\bl}=e^f(-h_{i\bj}h_{k\bl}-h_{i\bl}h_{k\bj}+2h_{i\bj}h_{k\bl})=e^f(h_{i\bj}h_{k\bl}-h_{i\bl}h_{k\bj}).$$
The symmetric curvature tensor
$\widetilde{K}_{i\bj k\bl}=0$.
So the holomorphic sectional curvature of $\tilde{\om}$ is zero, but the curvature is nonzero everywhere. As $\tilde{h}_{i\bj}=(1-|z|^2)\delta_{ij}+\bar{z}_iz_j$, $\tilde{\om}$ is not complete. The completion is $\overline{\mathbb{B}^n}$.
\end{exam}
\section{k-Gauduchon metrics}

In this section, we study the $k$-Gauduchon metrics on complex manifolds.

Recall that a Gauduchon metric (or called standard metric) is a Hermitian metric satisfying $\pa\bap \om^{n-1}=0$. On a compact complex manifold, there exist a unique Gauduchon metric  in the conformal class of any Hermitian metric \cite{G1}. For $1\leq k\leq n-1$, Fu-Wang-Wu consider the following equation in \cite{FW}
\begin{align} \label{equation}
 \pa\bap \om^{k}\wedge \om^{n-k-1}=0.\end{align}
A Hermitian metric is called $k$-Gauduchon if (\ref{equation}) is satisfied. We use the torsion 1-form  and operators on Hermitian manifolds to study $k$-Gauduchon metrics. First, direct computation gives
\begin{align}
\label{pak}
\pa\ov{\pa}\om^k\wg\om^{n-1-k}=\dfrac{k}{n-1}\pa\ov{\pa}\om^{n-1}-k(n-k-1)\pa\om\wg\ov{\pa}\om\wg\om^{n-3}.
\end{align}

\begin{lem}
Let $(M,\om)$ be a compact Hermitian manifold. Then\label{first}
	\begin{align*}
	*(\so\pa\bar{\pa}\om^{n-1})=(n-1)!(\hat{s}-s+|\pa^*\om|^2 ).
    \end{align*}
So $\om$ is Gauduchon if and only if $s-\hat{s}=|\pa^*\om|^2$.
\end{lem}
\begin{proof}
	As $*\om=\dfrac{\om^{n-1}}{(n-1)!}$, we have \begin{align*}
*\pa\ov{\pa}\om^{n-1}=(n-1)!\pa^*\ov{\pa}^*\om=(n-1)!\so\pa^*\tau.\end{align*}
Then it follows from (\ref{scal2}).
\end{proof}
We will assume $n\geq 3$ in the following discussion.
\begin{lem}
Let $(M,\om)$ be a compact Hermitian manifold. Then
 \label{second}
	\begin{align}
	*(\so\pa\om\wg\bar\pa\om\wg\frac{\om^{n-3}}{(n-3)!})=|\pa^*\om|^2-|\pa\om|^2,
	\end{align}
 or equivalently ,
\begin{align*}
	\Lambda^3(\sqrt{-1}\pa\om\wg\bar{\pa}\om)=6({|\pa^*\om|}^2-{|\pa\om|}^2).
	\end{align*}
	\end{lem}
	\begin{proof}
	
First, as $\La \pa\om=-\so \bap^*\om$ and $\La L\bap^*\om=(n-1)\bap^*\om$, we see the $(2,1)$ form $$\pa \om+\dfrac{\so}{n-1} L\bap^*\om$$ is primitive. By Proposition 6.29 in \cite{V}, for a primitive $\al\in\Om^{p,q}$ with $p+q=k,$
\begin{align}
\label{star}*\al=(-1)^p(\sqrt{-1})^{k^2}\frac{L^{n-k}\al}{(n-k)!}.
\end{align}
So		
	\begin{align}
	*\pa\om=\sqrt{-1}\pa\om\wg\frac{\om^{n-3}}{(n-3)!}-\frac{n-2}{(n-1)!}{\bar\pa}^*\om\wg\om^{n-2}-\frac{\so}{n-1}*L\bap^*\om
			\end{align}
By  $*L=\Lambda*$ and (\ref{star}),
\begin{align*}	
\frac{\sqrt{-1}}{n-1}*L{\bar\pa}^*\om
	=&\frac{1}{n-1}\Lambda({\bar\pa}^*\om\wg\frac{\om^{n-1}}{(n-1)!})\\
	=&{\bar\pa}^*\om\wg\frac{\om^{n-2}}{(n-1)!},
\end{align*}
where we use  (\ref{lr}) in the last equality.\\
Therefore,
\begin{align}
*\pa\om=&\sqrt{-1}\pa\om\wg\frac{\om^{n-3}}{(n-3)!}-{\bar\pa}^*\om\wg\frac{\om^{n-2}}{(n-2)!}.
\end{align}
Thus with $*\bap^*\om=\dfrac{\pa\om^{n-1}}{(n-1)!}$, we get
	\begin{align*}
	{|\pa\om|}^2dv=&-\sqrt{-1}\pa\om\wg\bar\pa\om\wg\frac{\om^{n-3}}{(n-3)!}-\pa\om\wg{\pa}^*\om\wg\frac{\om^{n-2}}{(n-2)!}\\
	=&-\sqrt{-1}\pa\om\wg\bar\pa\om\wg\frac{\om^{n-3}}{(n-3)!}+	{|\pa^*\om|}^2dv.
	\end{align*}
	Consequently,
	\begin{align}
	*(\sqrt{-1}\pa\om\wg\bar\pa\om\wg\frac{\om^{n-3}}{(n-3)!})={|\pa^*\om|}^2-{|\pa\om|}^2. \label{gaudu2}
	\end{align}
	\end{proof}
\noindent Then we have the following lemma. 	
\begin{lem}\label{3rd}
Let $(M,\om)$ be a compact Hermitian manifold and $k$ be an integer such that $1\leq k\leq n-1$. Then
\begin{align}
\notag&*(\so\pa\ov{\pa}\om^k\wg \om^{n-k-1})\\
=&k(n-2)!\left(\frac{k-1}{n-2}|\pa^*\om|^2+\frac{n-k-1}{n-2}|\pa\om|^2-s+\hat{s}\right)
\end{align}
\end{lem}
\begin{proof}
It is a combination of (\ref{pak}), Proposition \ref{scal2}, and Lemma \ref{second}.
\end{proof}

\begin{remk}
An integral version of the above two lemmas is given by Proposition 5.1 \cite{LY2}. We are curious about the pointwise equality and do the above calculation.  Similar formulas are also obtained in \cite{FZ} and \cite{IP}.
\end{remk}
Consequently, we have the following characterization.
\begin{prop}
Let $(M,\om)$ be a Hermitian manifold and $k$ be an integer such that $1\leq k\leq n-1$. Then the following are equivalent:
\begin{itemize}
\label{kg1}
\item[(1)] $\om$ is $k$-th Gauduchon;\\
\item[(2)] $\Lambda^{k+1}(\so\pa\bar\pa\om^k)=0$;\\
\item[(3)]$s-\hat{s}=\dfrac{k-1}{n-2}|\pa^*\om|^2+\dfrac{n-1-k}{n-2}|\pa\om|^2$.
\end{itemize}
\end{prop}
\begin{proof}
As $*\om^k=\dfrac{k!\om^{n-k}}{(n-k)!}$ (see e.g., \cite{Z}),  we have
\begin{align}\label{contrk}
\frac{1}{(k+1)!}\La^{k+1}(\so\pa\bap\om^k)=\so\pa\bap\om^k\wg\frac{\om^{n-k-1}}{(n-k-1)!}.
\end{align}
Then it follows from Lemma \ref{3rd}.

\end{proof}
\begin{cor} \label{cor4.6}
If $(M,\om)$ is $k$-Gauduchon, then $s\geq \hat{s}$. In particular, if $(M,\om)$ is pluriclosed or Astheno-K\"{a}hler, then $s\geq \hat{s}$.
\end{cor}

\begin{cor} \label{k-gau}
Let $(M,\om)$ be a compact Hermitian manifold and $k$ be an integer such that $1\leq k\leq n-1$. Then the following are equivalent:

\begin{itemize}
\item[(1)] $\om$ is $k$-Gauduchon for all $k, \ 1\leq k\leq n-1$;\\[-.2em]
\item[(2)] $|\pa\om|^2=|\pa^*\om|^2=\ag\pa\pa^*\om, \om\rg$.
\end{itemize}
\end{cor}
\noindent This follows from Lemma \ref{first} and Proposition \ref{kg1}.

\begin{remk}
Note that \begin{align*}
\so\pa\ov{\pa}\om^k\wg\om^{n-k-1}=k(\pa\ov{\pa}\om\wg\om^{n-2}+(k-1)\pa \om\wg\ov{\pa}\om\wg\om^{ n-3}).
\end{align*}
So for $1\leq p, q\leq n-1$, if $\om$ is $p$-Gauduchon and $q$-Gauduchon for $p\neq q$, then $\om$ is $k$-Gauduchon for all $k$.  \end{remk}

Using Proposition 4.5, we are able to obtain the following result which gives a slight generalization of Proposition 3.8 in \cite{IP}. 
\begin{prop} \label{k-gau2}
Let $(M,\om)$ be a compact Hermitian manifold and $k$ be an integer such that $1\leq k\leq n-2$. If $\om$ is locally conformal K\"ahler satisfying
\begin{align}
\label{conk}\int_M \La^{k+1}(\so \pa \bap \om^k)dv=0,
\end{align}
then $(M,\om)$ is K\"ahler.
\end{prop}
\begin{proof}
By (\ref{pak}) and (\ref{contrk}) (see also equality (5.1) in \cite{LY2}), for $1\leq k \leq n-2$ ,
\begin{align*}
 &\frac{1}{(k+1)!}\int_M\La^{k+1}(\so \pa \bap \om^k)dv\\[.3em]
=&-\dfrac{k}{(n-k-2)!}\int_M\so\pa\om\wg\bap\om\wg\om^{n-3}\\
=&\frac{(n-3)!k}{(n-k-2)!}(\|\pa\om\|^2-\|\pa^*\om\|^2).
 \end{align*}
If (\ref{conk}) is satisfied, then $$\|\pa\om\|^2=\|\pa^*\om\|^2=\|\tau\|^2.$$ As $\om$ is locally conformal K\"ahler,  then by (\ref{eta}) and  (\ref{lck}),
 \begin{align*}
\|\pa\om\|^2=\frac{1}{n-1}\|\tau\|^2.
\end{align*}
 So $\tau=0$ for  $n\geq 3$ and $\om$ is K\"ahler.
\end{proof}

\end{document}